\documentclass{article}
\usepackage[english]{babel}
\usepackage{amsrefs}
\usepackage{amssymb,amsmath,amsthm,amsfonts,epsfig,graphicx,psfrag,color}
\usepackage[utf8]{inputenc}
\usepackage{hyperref}
\hypersetup{
    colorlinks,
    citecolor=blue,
    filecolor=black,
    linkcolor=blue,
    urlcolor=magenta
}

\DeclareMathOperator{\Emb}{Emb}

\DeclareMathOperator{\Iso}{Iso}
\DeclareMathOperator{\mesh}{mesh}

\DeclareMathOperator{\comp}{comp}

\DeclareMathOperator{\D}{D_H}
\DeclareMathOperator{\DGH}{D_{\tilde H}}
\DeclareMathOperator{\DAM}{d_{GH^0}}

\theoremstyle{plain}
\newtheorem{thm}{Theorem}[section]
\newtheorem{cor}[thm]{Corollary}
\newtheorem{lem}[thm]{Lemma}
\newtheorem{prop}[thm]{Proposition}

\theoremstyle{definition}

\newtheorem{rmk}[thm]{Remark}

\theoremstyle{remark}

\DeclareMathOperator{\diam}{diam}

\DeclareMathOperator{\dist}{dist}

\newcommand{\DS}{\mathcal{DS}}
\newcommand{\Ur}{\mathbb U}
\newcommand{\U}{\mathcal U}

\newcommand{\G}{\mathcal G}
\newcommand{\F}{\mathbb F}
\newcommand{\V}{\mathbb V}
\newcommand{\R}{\mathbb R}

\newcommand{\K}{\mathcal K}

\newcommand{\Z}{\mathbb Z}
\newcommand{\N}{\mathbb N}

\renewcommand{\epsilon}{\varepsilon}

\begin{document}

\author{Alfonso Artigue\footnote{Email: artigue@unorte.edu.uy. Adress: Departamento de Matemática y Estadística del Litoral, Universidad de la Rep\'ublica, Gral. Rivera 1350, Salto, Uruguay.}}
% \title{Gromov-Hausdorff Generic Dynamics}
% \title{Generic Dynamics Up to Isometric Conjugacies}
\title{Generic Dynamics on Compact Metric Spaces}

\date{\today}
\maketitle

\begin{abstract}
We prove that generically and
modulo a topological conjugacy there is only one 
dynamical system.
% homeomorphism in 
% the Gromov-Hausdorff space of dynamical systems introduced by Arbieto and Morales.
\end{abstract}

\section{Introduction}

In this article we will study generic dynamical systems. 
A property in a complete metric space is \emph{generic} 
if the set of the points satisfying this property is residual.
Let us explain the space of dynamical systems that we will consider.

To define a dynamical system we need a phase space. 
For this purpose we will consider compact metric spaces. A key point of the paper is that no phase space is fixed. 
Instead, 
given a complete metric space $(\V,\dist)$ we consider all the compact subsets of $\V$. 
Considering the results in this paper, we suggest to think that $\V$ is Euclidean $\R^n$ or Urysohn universal space 
$\Ur$ (that is defined in \S\ref{secSS}).
We remark that $\Ur$ contains an isometric copy of each compact metric space.
The idea is that, when performing a perturbation of a dynamical system, 
we are allowed to perturb its domain too.

When a compact subset $M\subset\V$ is fixed, a \emph{dynamical system} is a homeomorphism of this set,
a continuous surjective map or a set-valued map. 
A function $f\colon M\to M$ will be identified with its graph $\{(x,f(x)):x\in M\}$. 
If $M$ is compact and $f$ is continuous then 
its graph is a compact subset of $M\times M$.
Thus, an arbitrary compact subset of $M\times M$ can be regarded as 
a \emph{generalized continuous function}. In this way we arrive to set valued maps. 
If $f\subset M\times M$ is compact, then we define $f(x)=\{y\in M:(x,y)\in f\}$. 
Therefore, in this paper, a \emph{dynamical system} is a compact subset of $M\times M$, for some compact subset $M\subset \V$. 
There is a technical detail, we assume that $f\subset M\times M$ and its inverse are onto. 
The inverse of $f$ is $f^{-1}=\{(y,x)\in M\times M:(x,y)\in f\}$ 
and $f$ is \emph{onto} if for all $y\in M$ there is $x\in M$ such that $(x,y)\in f$. 
The set of dynamical systems, that depends on $\V$, will be denoted as $\DS(\V)$.

In order to define and study generic properties on $\DS(\V)$, we need a topology for this set. 
Given that we defined dynamical systems as 
certain compact subsets of $\V\times \V$, 
it is natural to consider the Hausdorff metric. It will be denoted as $\D$, and in this way 
we obtain the metric space of dynamical systems 
$$(\DS(\V),\D),$$ 
which is the object of study of the present paper.

To state our main result suppose that $(\V,\dist)$ is Polish and perfect (as for example $\R^n$ and $\Ur$). 
We remark that every surjective map, in particular every homeomorphism, 
of a compact metric space has an isometric conjugate in $(\DS(\Ur),\D)$ (where $\Ur$ is the Urysohn universal space).
In Theorem \ref{thmGenDSSH} we will show that generically and
modulo a topological conjugacy there is only one dynamical system in $\DS(\V)$. 
In other words, we show that there is a single homeomorphism in $\DS(\V)$ whose conjugacy class is a dense $G_\delta$ subset of $\DS(\V)$.
This dynamical system is a homeomorphism of a Cantor space, known as the \emph{Special Homeomorphism} \cite{AGW}.

From the following viewpoint, this result is quite natural. 
On one hand: 

\begin{center}
 \emph{a generic compact metric space is a Cantor space.} 
\end{center}
This is true if, for instance, we consider the space of compact subsets of a complete and perfect metric space (see \cite{Wie}*{the paragraph below Lemma 1}). 
It is also true with respect to the Gromov-Hausdorff metric \cite{Rouyer}*{Corollary 5}.
On the other hand: 
\begin{center}
 \emph{on a fixed Cantor space, a generic homeomorphism is conjugate to the Special Homeomorphism.}
\end{center}
This result was first proved by Kechris and Rosendal \cite{KeRo}. 
In \cite{AGW}, Akin, Glasner and Weiss gave a concrete construction.
Thus, on a generic compact metric space (= \emph{a Cantor space}) a generic homeomorphism is conjugate to the Special Homeomorphism. 
What we prove in Theorem \ref{thmGenDSSH} is: 
\begin{center}
 \emph{a generic (homeomorphism of a compact metric space) is conjugate to the Special Homeomorphism.}
\end{center}
The key of our approach is the space $(\DS(\V),\D)$ which allows us to measure the distance between 
dynamical systems defined on different spaces.

A sketch of the proof of Theorem \ref{thmGenDSSH} is as follows. First, we show in Proposition \ref{propGenCantorHomeo} 
that a generic $f\in\DS(\V)$ is a homeomorphism of a Cantor space. 
From this result it follows the density of the conjugacy class of the Special Homeomorphism in $\DS(\V)$. 
To prove the genericity we consider a map 
$$\varphi\colon \Emb(K,\V)\times H(K)\to\DS_{CH}(\V)$$ 
defined as $\varphi(h,f)=h f h^{-1}$ (a homeomorphism of $h(K)$). 
There, $K$ is a fixed Cantor space, $\Emb(K,\V)$ is the space of continuous injective maps from $K$ to $\V$, $H(K)$ is the space of homeomorphisms of $K$ and 
$\DS_{CH}(\V)\subset\DS(\V)$ is the set of homeomorphisms of Cantor spaces. 
If $\G$ denotes the set of homeomorphisms $f\in H(K)$ conjugate to the Special Homeomorphism, 
we will show that $\varphi(\Emb(K,\V)\times \G)$ is a dense $G_\delta$ subset of $\DS_{CH}(\V)$, which easily finishes the proof.

Some dynamical properties of the Special Homeomorphism are known, for instance,
it has vanishing topological entropy \cite{GlWe}, the pseudo-orbit tracing property \cite{BeDa} 
and that it is not conjugate to a subshift \cite{Sears}. 
The reader may see \cite{Akin2016} for a survey and more results on dynamical systems on Cantor spaces.
In Corollary \ref{corMani} we give 
an application of Theorem \ref{thmGenDSSH} to the dynamics of homeomorphisms of compact manifolds.

In \cite{Ho}, Hochman considered another extension of the genericity of the Special Homeomorphism. 
In this paper it is proved, among other results, the genericity of the Special Homeomorphism 
in the space of subshifts of the Hilbert cube.  
This space is \emph{topologically universal} for dynamical systems, 
i.e., it contains a topological conjugate of every dynamical system 
of a compact metric space (see \cite{Ho}*{\S 2.5}). 
We say that a homeomorphism 
$\varphi\colon X\to Y$ is a topological conjugacy (or simply, conjugacy) between 
two maps $f\colon X\to X$ and $g\colon Y\to Y$ if $g \varphi=\varphi  f$. 
If, in addition, $\varphi$ is an isometry we say that it is an isometric conjugacy.
The \emph{isometric universality} of 
the space $\DS(\Ur)$ 
is the key for our next result. 

The second result that we obtained, Theorem \ref{thmGHGen}, 
is the projection of Theorem \ref{thmGenDSSH} 
onto the space of dynamical systems modulo isometric conjugacies.
If $\K(\V)$ denotes the set of compact subsets of $\V$ then 
it is natural to consider the quotient space 
$$\K_{\Iso}(\V)=\K(\V)/\sim$$
where $\sim$ is the equivalence relation of isometry (between compact subsets of $\V$).
For $\V=\Ur$ we obtain the so called \emph{Gromov-Hausdorff space}.
In a recent paper, Arbieto and Morales \cite{AM} extended this idea to dynamical systems. 
If in $\DS(\V)$ we consider the equivalence relation $\sim$ of isometric conjugacies (see \S\ref{secIsoConj}) we obtain 
the \emph{space of dynamical systems} 
\[
 \DS_{\Iso}(\V)=\DS(\V)/\sim.
\]
In \cite{AM} the authors studied the notion of topological stability, defining the topology of $\DS_{\Iso}(\V)$ 
 via a quasi-metric (see Remark \ref{rmkAMmetrica}). 
In this paper we give a natural metric for this space. 
It is defined in \S\ref{secQuotSD}
and the compatibility with Arbieto-Morales metric is proved in Theorem \ref{thmCompAm}. 
In Theorem \ref{thmGHGen}
we will show that if $(\V,\dist)$ is Polish, perfect and ultrahomogeneous 
then a generic $\tilde f\in \DS_{\Iso}(\V)$ is conjugate to the Special Homeomorphism. 
A space is \emph{ultrahomogeneous} if every isometry between compact subsets can be extended to a global isometry, 
see \S\ref{secSS}.

It could not be clear why we do not fix $\V=\Ur$. 
On one hand, since $\Ur$ is universal we are considering all the 
compact metric spaces. But, on the other hand, it is more general not to fix $\V$. 
Let us say more. During the preparation of this work we found that 
the key property (for our purposes) of $\Ur$ is the ultrahomogeneity, not the universality. 
Besides, allowing $\V=\R^n$ (ultrahomogeneous too) we are including a quite standard space. 
The Gromov-Hausdorff distance can be defined in $\R^n$ (allowing only isometries of $\R^n$). 
This space is studied, for example, by M\'emoli in \cite{Me08} where also some applications are given.

As we said, in the spaces of dynamical systems we will consider set-valued maps. 
It could be natural to consider only continuous maps or homeomorphisms, 
but this would imply that the metric $\D$ would not be complete (see Remark \ref{rmkNoCompleto}). 
In fact, we will show that the metric completion of the space of homeomorphisms is 
the space of set-valued dynamics. That is, 
$(\DS(\V),\D)$ is complete (Proposition \ref{propDSComplete})
and homeomorphisms are dense in $\DS(\V)$ (Proposition \ref{propExpHomDens}). 
These results assume that $(\V,\dist)$ is complete and perfect.

I thank Mauricio Achigar, Ignacio Monteverde and José Vieitez for useful conversations 
during the preparation of this work.
Also, I thank the referee for calling my attention to \cite{Akin2004}*{Theorem 1.2}, which allowed me to improve the proof 
of Theorem \ref{thmGenDSSH}. 

\section{Dynamical systems}
\label{secDSP}
In this section we give some preliminaries on metric spaces and 
the basic properties of $\DS(\V)$. 
In \S\ref{secSH} we prove Theorem \ref{thmGenDSSH}.

% 
% In \S\ref{secMS} we give some preliminaries on metric spaces.
% In \S\ref{secDS} we consider the space $\DS(\V)$ of dynamical systems 
% on compact subsets of the metric space $(\V,\dist)$.
% In \S\ref{secCompDS} we show that
% if $(\V,\dist)$ is complete and perfect then 
% a generic $f\in\DS(\V)$ is a homeomorphism of a Cantor space. 

\subsection{Cantor spaces and the Hausdorff distance}
\label{secMS}

We say that a topological space is \emph{perfect} if it has no isolated points.
A metric space is a \emph{Cantor space} if it is compact, perfect and totally disconnected. 
% \rojo{The next result is similar to \cite{Akin2004}*{Proposition 2.1}.}
\begin{prop}[\cite{Akin2004}*{Proposition 2.1}]
\label{propCantorChico} Any two Cantor spaces are homeomorphic. 
Every open subset of a complete and perfect space contains a Cantor space.
% \footnote{To show the second statement see the proof 
% of \cite{Akin2004}*{Proposition 2.1}.}
\end{prop}

For an arbitrary metric space $(N,\dist)$ denote by $\K(N)$ the set of compact subsets of $N$.
The Hausdorff distance between $X,Y\in\K(N)$ 
is defined as
\begin{equation}
\label{ecuHausd}
 \dist_H(X,Y)=\inf\{\epsilon>0:X\subseteq B_\epsilon(Y)\text{ and }Y\subseteq B_\epsilon(X)\},
\end{equation}
where $B_\epsilon(x)=\{y\in N:\dist(x,y)<\epsilon\}$ and $B_\epsilon(X)=\cup_{x\in X} B_\epsilon(x)$.
% The pair $(\K(N),\dist_H)$ is a metric space called \emph{hyperspace} \cite{IN}.

\begin{prop}[\cites{IN,Akin2004}]
% *{Exercise 2.15} ahi está la completitud 
\label{propHyperCompleto}
% (Parecería ser de Hahn 1932)
If $(N,\dist)$ is complete (compact) then $(\K(N),\dist_H)$ is complete (compact).
\end{prop}

Fix a Cantor space $K\subset \V$ and let $H(K)$ be the space of homeomorphisms of $K$. 
Suppose that $h\colon A\to B$ is a function with $A,B\subset \V$.
In this case we define 
$$\|h\|=\sup_{a\in A}\dist(a,h(a)).$$
Let $H(A,B)$ be the set of homeomorphisms $h\colon A\to B$.
For a family $\U$ of subsets of $\V$ let $\mesh(\U)=\sup\{\diam(U):U\in\U\}$.
The next result gives a particular way to calculate the Hausdorff distance between Cantor spaces.
\begin{prop}
\label{propCantorCerca}
 If $A,B\subset \V$ are Cantor spaces then 
 $$\dist_H(A,B)=\inf_{h\in H(A,B)}\|h\|.$$ 
\end{prop}

\begin{proof}
The inequality $\leq$ is clear (and holds for arbitrary $A$ and $B$ homeomorphic).
Define $\rho=\dist_H(A,B)$ and for $\delta>0$ let $\epsilon=3\delta+\rho$.
% 
% Idea 1:
% 
% For some $n$ large, there are two sets with $n$ points $\{a_1,\dots,a_n\}\subset A$ 
% and $\{b_1,\dots,b_n\}\subset B$ such that 
% $\dist(a_i,b_i)<\rho+\delta$, $A=\cup_{i=1}^nB_\delta(a_i)$ and $B=\cup_{i=1}^nB_\delta(b_i)$.
% 
% Idea 2: 
Let $\U_A$ be a clopen partition of $A$ with $\mesh(\U_A)<\delta$. 
Take an injective function $f_1\colon \U_A\to B$ such that for each $U\in\U_A$ there is $x_U\in U$ with $\dist(x_U,f_1(U))<\rho+\delta$. 
For each $U\in \U_A$ consider a clopen set $U'\subset B$ such that 
$f_1(U)\in U'$, $\U^1_B=\{U':U\in\U_A\}$ are pairwise disjoint and $\mesh(\U^1_B)<\delta$.

Suppose that $\U^1_B$ does not cover $B$. 
Let $\U^2_B$ be a clopen partition of $B\setminus \cup \U^1_B$ with $\mesh(\U^2_B)<\delta$. 
Take an injective function $f_2\colon \U^2_B\to A$ such that 
for each $V\in\U^2_B$ there is $y_V\in V$ with 
$\dist(y_V,f_2(V))<\rho+\delta$.
Let $\U'_B=\U^1_B\cup\U^2_B$ and 
define $f_3\colon \U'_B\to \U_A$ as 
$f_3(U')=U$ for all $U'\in\U^1_B$ and
if $V\in \U^2_B$ then $f_3(V)\in\U_A$ contains $y_V$.
Consider a clopen partition $\U'_A$ of $A$, refining $\U_A$, 
admitting $f_4\colon \U'_B\to \U'_A$ bijective with $f_4(V)\subset f_3(V)$ for all $V\in \U'_B$. 
Let $h\colon B\to A$ be a homeomorphism such that $h(V)=f_4(V)$ for all $V\in \U'_B$.
Since $\mesh(\U'_A),\mesh(\U'_B)<\delta$, and for each $V\in \U'_B$ there are $x\in V$ and $y\in h(V)$ such that 
$\dist(x,y)<\rho+\delta$ we conclude that $\|h\|<\rho+3\delta=\epsilon$. 
% Suppose that $\U_B$ covers $B$. 
% Since each $U\in \U_A$ and $U'\in \U_B$ are Cantor spaces, by Lemma \ref{lemCantorChico}, 
% there is a homeomorphism $h_U\colon U\to U'$. 
% Define $h\colon A\to B$ as $h(x)=h_U(x)$ whenever $x\in U$.
% For $y\in A$, take $U\in\U_A$ containing $y$, and we have
% \[
% \begin{array}{ll}
%   \dist(y,h(y))&\leq\dist(y,x_U)+\dist(x_U,f_1(U))+\dist(f_1(U),h(y))\\
%   &<\delta+(\rho+\delta)+\delta=\epsilon.
% \end{array}
% \]
% This proves that $\|h\|<\epsilon$.
% 
\end{proof}

\subsection{Polish spaces}

A topological space is said to be \emph{Polish} if it is separable and admits a compatible complete metric.
A set $R\subset X$ is $G_\delta$ if it is a countable intersection of open subsets of $X$.
\begin{thm}[\cite{Akin2004}*{Theorem 1.2}]
\label{thmTopFubini}
Let $\varphi\colon X\to Y$ be a continuous open surjection with $X$ and $Y$ Polish spaces. 
If $R$ is a dense $G_\delta$ subset of $X$ then 
\[
 \varphi_\#(R)=\{y\in Y:\varphi^{-1}(y)\cap R\text{ is dense in }\varphi^{-1}(y)\}
\]
is a dense $G_\delta$ subset of $Y$.
%  Let $X,Y$ be metric spaces.
%  If $\varphi\colon X\to Y$ is continuous, open and surjective and $R\subset X$ is a dense $G_\delta$ subset then 
%   $\varphi(R)\subset Y$ is a dense $G_\delta$ subset.
 \end{thm}

From this result we deduce the following consequence which is the key for Theorems \ref{thmGenDSSH} and \ref{thmGHGen}.
 
\begin{cor}
\label{corTopFubini}
Let $\varphi\colon X\to Y$ be a continuous open surjection with $X$ and $Y$ Polish spaces. 
If $R$ is a dense $G_\delta$ subset of $X$ and $R=\varphi^{-1}(\varphi(R))$ then $\varphi(R)$ is a 
dense $G_\delta$ subset of $Y$.
\end{cor}

\begin{proof}
 Since $R=\varphi^{-1}(\varphi(R))$, we have that $y\in\varphi(R)$ if and only if $\varphi^{-1}(y)\subset R$ 
 if and only if $\varphi^{-1}(y)\cap R\neq\emptyset$. 
  This proves that $\varphi(R)=\varphi_\#(R)$. 
  Thus, the result follows by Theorem \ref{thmTopFubini}.
\end{proof}

The next classical result allows us to conclude that a $G_\delta$ subset of a Polish space is Polish.

\begin{thm}[Alexandroff’s theorem \cite{Oxtobi}*{Theorems 12.1 and 12.3}]
\label{thmAlex}
In a complete space, a subset
is $G_\delta$ if and only if it is completely metrizable.
\end{thm}

\subsection{Spaces of dynamical systems}
\label{secDS}
Let $(\V,\dist)$ be a metric space.
On $\V^2=\V\times\V$ we consider the max-metric $\dist_2\colon\V^2\times\V^2\to\R$ given by 
\begin{equation}
 \label{ecuDistMax}
 \dist_2((a,b),(c,d))=\max\{\dist(a,c),\dist(b,d)\}.
\end{equation} 
Consider the projections $\pi_i\colon \V^2\to\V$ given by $\pi_i(a_1,a_2)=a_i$ for $i=1,2$. 
Define the set 
\[
 \DS(\V)=\{f\in\K(\V^2):\pi_1(f)=\pi_2(f)\}.
\]
A \emph{dynamical system} is any $f\in\DS(\V)$ (i.e., a surjective relation). 
Denote by $\D$ the Hausdorff distance of $\K(\V^2)$ associated to $\dist_2$.
For $f\in\DS(\V)$ define $M(f)=\pi_i(f)$, and we say that $f$ is a dynamical system on $M(f)$. 
Note that since the projections are continuous and $f$ is compact, we have that $M(f)$ is compact.
In $\DS(\V)$ there are three classes: 
\begin{enumerate}
 \item homeomorphisms of $M\subset \V$,
 \item surjective continuous maps of $M\subset\V$, 
 \item set valued maps of $M\subset \V$, where for each $x\in M$, $f(x)=\{y\in M:(x,y)\in f\}$ may not be a singleton.
\end{enumerate}
In any case, if $f\in\DS(\V)$ we define $f^{-1}=\{(y,x)\in\V^2:(x,y)\in f\}\in\DS(\V)$.

\begin{rmk}
\label{rmkDistDS}
The distance between $f,g\in\DS(\V)$, $\D(f,g)$, by definitions 
\eqref{ecuHausd} and \eqref{ecuDistMax} is the infimum $\epsilon>0$ 
satisfying the following conditions:
\begin{itemize}
 \item for all $(x,y)\in f$ there is $(x',y')\in g$ such that $\dist(x,x')<\epsilon$ and $\dist(y,y')<\epsilon$,
 \item for all $(x,y)\in g$ there is $(x',y')\in f$ such that $\dist(x,x')<\epsilon$ and $\dist(y,y')<\epsilon$.
\end{itemize}
\end{rmk}

\begin{rmk}
 Since $(x,y)\mapsto (y,x)$ is an isometry of $\V^2$, we have that
$\D(f,g)=\D(f^{-1},g^{-1})$ for all $f,g\in\DS(\V)$.
\end{rmk}

If $M,N\in\K(\V)$ and $f,g\colon M\to N$ are continuous functions,
the $C^0$-\emph{distance} is defined as
\begin{equation}
 \label{ecuDistC0}
 \dist_{C^0}(f,g)=\sup_{x\in M}\dist(f(x),g(x)).
\end{equation}
Let $C(M,N)$ be the set of all the continuous functions $f\colon M\to N$.

% The next result is from \cite{Akin}*{Proposition 20, Chapter 7, p. 135}. 
% Since its proof is short we include it. 

\begin{prop}[\cite{Akin}*{Proposition 20, Chapter 7, p. 135}]
\label{propCompDyC0}
The metric $\D$ is compatible with the $C^0$-topology of $C(M,N)$.
\end{prop}

% \rojo{See  for a proof of Proposition \ref{propCompDyC0}.}

% \begin{proof}
%  First note that $\D\leq \dist_{C^0}$. 
%  Now suppose that $\D(f_n,f)\to 0$. 
% Since $f$ is continuous and $M$ is compact, 
% $f$ is uniformly continuous and for all $\epsilon>0$ there is $\delta>0$ 
% such that if $\dist(x,y)<\delta$ then $\dist(f(x),f(y))<\epsilon/2$. 
% Take a positive $r<\min\{\delta,\epsilon/2\}$. 
% Suppose that $\D(f_n,f)<r$. We will show that $\dist_{C^0}(f_n,f)<\epsilon$. 
% For $x\in M$, by Remark \ref{rmkDistDS} there is $x'\in M$ such that 
% $\dist(x,x')<r$ and $\dist(f(x'),f_n(x'))<r$. 
% Since $\dist(x,x')<r$ we have that $\dist(f(x),f(x'))<\epsilon/2$. 
% By the triangular inequality we conclude $\dist(f(x),f_n(x))<r+\epsilon/2<\epsilon$. 
% This finishes the proof.
% \end{proof}

\begin{rmk}
\label{rmkNoCompleto}
It is well known that the space $C(M,M)$ with $\dist_{C^0}$ is complete. 
By Proposition \ref{propCompDyC0}, $\D$ is compatible with the topology induced by $\dist_{C^0}$. 
Let us show that $C(M,M)$ is not complete with respect to $\D$. 
Consider $M$ as the interval $[0,1]$ (or an isometric copy in $\V$) and define 
$f_n(x)=x^n$ for all $n\in\N$ and all $x\in [0,1]$. 
Define $f\subset [0,1]\times [0,1]$ as $f=([0,1]\times\{0\})\cup(\{1\}\times [0,1])$.
It is easy to see that $\D(f_n,f)\to 0$. 
Thus, $f_n$ is a Cauchy sequence with respect to $\D$. 
Its limit exists in $\DS(\V)$, is the set-valued function $f$, but $f$ is not a function. 
This example explains why we consider non-injective maps and set-valued functions, even if our interest is on 
homeomorphisms, the set-valued functions completes the space of homeomorphisms (with respect to $\D$, which is 
the metric that allows us to measure the distance between functions on different spaces).
\end{rmk}

\begin{rmk}
\label{rmkDistsChotas}
For all $f,g\in \DS(\V)$ it holds that $\dist_H(M(f),M(g))\leq\D(f,g)$.
Given $X\in\K(\V)$ denote by $i_X$ the identity map of $X$. 
 It holds that $\dist_H(X,Y)=\D(i_X,i_Y)$ for all $X,Y\in\K(\V)$.
\end{rmk}

Since Cantor homeomorphisms play a key role we introduce the following notation
\[
\DS_{CH}(\V)= \{f\in\DS(\V):f\text{ is a homeomorphism of a Cantor space}\}.
\]
The next result extends Proposition \ref{propCantorCerca}, and will be used to estimate the distance between Cantor homeomorphisms. 

\begin{prop}
\label{propCantorCercaDos}
 If $g,j\in\DS_{CH}(\V)$ then $\D(g,j)<\delta$ if and only if there are homeomorphisms
 $h_1,h_2\colon M(g)\to M(j)$ such that $\|h_1\|,\|h_2\|<\delta$ and $h_2  g=j  h_1$.
\end{prop}

\begin{proof}
Assume that $\D(g,j)<\delta$. 
By Proposition \ref{propCantorCerca}, there is a homeomorphism $h\colon g\to j$ with $\|h\|<\delta$. 
Let $\pi_i\colon M(j)\times M(j)\to M(j)$ be the canonical projections, for $i=1,2$, and 
define 
$h_1(x)=\pi_1(h(x,g(x))$ and $h_2(g(x))=\pi_2(h(x,g(x)))$. 
That is, $h(x,g(x))=(h_1(x),h_2(g(x)))\in j$ and $j(h_1(x))=h_2(g(x))$.
Finally note that $\|l_i\|\leq\|h\|$ for $i=1,2$.
\end{proof}

We give two elementary results that will be used in \S\ref{secCompDS}.

\begin{prop}
\label{propFinRelDense}
Finite relations are dense in $\DS(\V)$.
\end{prop}

\begin{proof}
Given $f\in\DS(\V)$ and $\epsilon>0$, as $M(f)$ is compact, we can take a finite set $Q\subset M(f)$ with $\dist_H(Q,M(f))<\epsilon$. 
Define $g\subset Q^2$ by 
$(a,b)\in g$ if and only if 
there is $(x,y)\in f$ with $\dist_2((x,y),(a,b))<\epsilon$ (i.e., $\dist(x,a)<\epsilon$ and $\dist(y,b)<\epsilon$). 
By the definition of $g$ we have that $\D(f,g)<\epsilon$ and $M(g)=Q$ is finite.
\end{proof}

\begin{prop}
\label{propCharFun}
If $f\in\DS(\V)$ and $f(x)=\{y\in M(f):(x,y)\in f\}$ is a singleton for all $x\in M(f)$ then $f$ is a continuous function. 
\end{prop}

\begin{proof}
It is clear that $f$ is a function. 
Given a convergent sequence $x_n\in M(f)$ if $f(x_n)$ has an accumulation point $y\neq f(x)$ then, 
$f$ is compact, $(x,f(x)),(x,y)\in f$. Which contradicts our hypothesis.
\end{proof}

\subsection{On a complete and perfect space}
\label{secCompDS}
We start assuming that $(\V,\dist)$ is complete.
This implies that $\V^2$ is complete with the max-metric \eqref{ecuDistMax}.

\begin{prop}
\label{propDSComplete}
If $(\V,\dist)$ is complete then $(\DS(\V),\D)$ is complete. 
\end{prop}

\begin{proof}
Since $\V^2$ is complete we have that $(\K(\V^2),\D)$ is complete (Proposition \ref{propHyperCompleto}). 
It remains to show that $\DS(\V)$ is closed in $\K(\V^2)$.
Suppose that $f_n$ is a sequence in $\DS(\V)$ converging to $f\in\K(\V^2)$. 
Since the projections $\pi_i$ are continuous, $\pi_i(f_n)$ converges to $\pi_i(f)$. 
Therefore, $\pi_1(f)=\pi_2(f)$ and $f\in\DS(\V)$.
\end{proof}

% 
% 
% 
% \begin{proof}
%  Since $\V$ is perfect we can take $y\in B_{\epsilon/4}(x)\setminus\{x\}$. 
%  Define $C_0=\{x,y\}$. Consider $0<\epsilon_1<\min\{\dist(x,y)/2,\epsilon/8\}$ 
%  and two points $x_1\in B_{\epsilon_1}(x)\setminus\{x\}$ and 
%  $y_1\in B_{\epsilon_1}(y)\setminus\{y\}$. Define $C_1=\{x,x_1,y,y_1\}$. 
%  Around each point of $C_1$ take pairwise disjoint balls of radius $\epsilon_2<\epsilon/16$. 
%  On each of these balls take a point different from its center and define $C_2$ as the union of $C_1$ with these new 4 points. 
%  In this way we have a sequence of finite sets $C_n\subset B_{\epsilon/2}(x)$, $n\in\N$. 
%  Since $\V$ is complete, $C_n$ converges to a Cantor space that is contained in the closure of $B_{\epsilon/2}(x)$. 
% \end{proof}

The next result 
is known in the space of homeomorphisms of a fixed Cantor space \cites{Sh89,Kimura}.

\begin{prop}
\label{propExpHomDens}
If $(\V,\dist)$ is complete and perfect then 
the set of subshifts of finite type on Cantor spaces is
% $\ExpTD$ is 
dense in $\DS(\V)$.
\end{prop}

\begin{proof}
Given $f\in\DS(\V)$, by Proposition \ref{propFinRelDense}, we know that there is a finite relation 
$g\in\DS(\V)$ that is $\D$-close to $f$.
Suppose that $M(g)=\{q_1,\dots,q_k\}$.
Let $A$ be the $k\times k$ matrix defined by 
$A_{ij}=1$ if $(q_i,q_j)\in g$ and $A_{ij}=0$ otherwise. 
This matrix induces a subshift of finite type in the set of symbols $Q=\{q_1,\dots,q_k\}$. 
Let $\Sigma=\{a\in Q^\Z: (a_i,a_{i+1})\in g\text{ for all }i\in\Z\}$. 
Since $\Sigma$ may not be perfect, consider $\Sigma_*=\Sigma\times\{0,1\}^\Z$.
For each $q\in Q$ define the Cantor space $\Sigma_q=\{(a,b)\in\Sigma_*:a_0=q\}$. 

Suppose $\epsilon>0$ is given. By Proposition \ref{propCantorChico}, for each $q\in Q$ there is a homeomorphism onto its image 
$\varphi_q\colon \Sigma_*\to\V$ 
such that $\varphi(\Sigma_q)\subset B_\epsilon(q)$. 
Define $N=\varphi(\Sigma_*)$ and 
the homeomorphism $h\colon N\to N$ as $h=\varphi \sigma \varphi^{-1}$, 
where $\sigma\colon \Sigma_*\to\Sigma_*$ is the shift map.
It is clear that $M(h)=N$ and $D(h,g)<\epsilon$. 
Since $g$ is conjugate to the subshift of finite type $\sigma$ on $\Sigma_*$ the proof ends.
\end{proof}

\begin{rmk}
Since subshifts of finite type are homeomorphisms, by Proposition \ref{propExpHomDens} we have that 
if $(\V,\dist)$ is complete and perfect then homeomorphisms are dense in $\DS(\V)$. 
Thus, by Proposition \ref{propDSComplete}, we have that the space of set valued dynamical systems is the completion of 
the space of homeomorphisms (with respect to $\D$).
\end{rmk}

% On a topological space a set is $G_\delta$ if it is a countable intersection of open sets.
% A property is \emph{generic} if it holds in a residual set (a set containing a dense $G_\delta$ set).
\begin{prop}
\label{propGenCantorHomeo}
If $(\V,\dist)$ is complete and perfect then 
$\DS_{CH}(\V)$ is a dense $G_\delta$ subset of $\DS(\V)$.
\end{prop}

\begin{proof}
We start showing that the set 
\[
 \{f\in\DS(\V):f\text{ is a homeomorphism}\}
\]
is a dense $G_\delta$ subset of $\DS(\V)$.
For $\epsilon>0$ given define 
$$ A_\epsilon=\{f\in\DS(\V):\diam(f(x))<\epsilon,\,\forall x\in M(f)\}.$$ 
By Proposition \ref{propExpHomDens} we know that homeomorphisms are dense in $\DS(\V)$. 
In particular, for each $\epsilon>0$, $ A_\epsilon$ is dense in $\DS(\V)$. 
Let us show that $\DS(\V)\setminus A_\epsilon$ is closed. 
Suppose that $f_n$ is a convergent sequence in $\DS(\V)$ such that 
for some $x_n\in M(f_n)$ we have $\diam(f_n(x_n))\geq\epsilon$ for all $n\in\N$. 
Let $f$ be the limit of $f_n$.
The limit (in the Hausdorff metric) of $f_n(x_n)$ is contained in $f(x)$, and we conclude that $\diam(f(x))\geq \epsilon$. 
This proves that $ A_\epsilon$ is open. 
By Proposition \ref{propCharFun} we know that $f\in\DS(\V)$ is a homeomorphism if and only if $f,f^{-1}\in A_\epsilon$ for all $\epsilon>0$. 
For $n\in\N$, let $U_n$ be the open and dense set of dynamical systems $f$ such that $f,f^{-1}\in A_{1/n}$. 
The set $\cap_{n\in\N}U_n$ is the set of the homeomorphisms $f\in\DS(\V)$.

To conclude the proof we will show that the set
\begin{equation}
\{f\in\DS(\V):M(f)\text{ is a Cantor space}\}  
\end{equation}
is a dense $G_\delta$ subset of $\DS(\V)$.
Given $M\in\K(\V)$ define 
$$\|M\|=\sup_{x\in M} \diam(\comp_x(M)),$$
where $\comp_x(M)$ denotes the connected component of $M$ containing $x$.
For $\epsilon>0$ let $$ A_\epsilon=\{f\in\DS(\V):\|M(f)\|<\epsilon\}.$$
It is easy to see that $ A_\epsilon$ is open in $\DS(\V)$, for each $\epsilon>0$. 
From Proposition \ref{propExpHomDens}
we have that $ A_\epsilon$ is dense in $\DS(\V)$.
Define 
\[
  A_0=\{f\in\DS(\V):\|M(f)\|=0\}.
\]
That is, $f\in A_0$ if and only if $M(f)$ is totally disconnected.
Since $A_0=\cap_{n\geq 1} A_{1/n}$ we have that $ A_0$ is a dense $G_\delta$ set. 

By Proposition \ref{propExpHomDens}
we know that homeomorphisms on perfect sets 
are dense in 
$\DS(\V)$. 
For $\epsilon>0$ define
\[
  F_\epsilon=\{f\in\DS(\V): \exists x\in M(f)\text{ s.t. } B_\epsilon(x)=\{x\}\},
\]
where $B_\epsilon(x)$ is the open ball in $M(f)$. Also consider
\[
 F_0=\{f\in\DS(\V):M(f)\text{ is perfect}\}.
\]
It is easy to see that $ F_\epsilon$ is closed in $\DS(\V)$. 
Consequently, $\DS(\V)\setminus F_\epsilon$ is open and dense in $\DS(\V)$. 
Since 
\[
 F_0=\cap_{n\geq 1}(\DS(\V)\setminus F_{1/n})
\]
we have that $F_0$ is a dense $G_\delta$ set. 

Therefore $F_0\cap A_0$ is a dense $G_\delta$ set. 
Note that $f\in F_0\cap A_0$ if and only if $M(f)$ is a Cantor space.
\end{proof}

\subsection{The Special Homeomorphism}
\label{secSH}

In this section we assume that $(\V,\dist)$ is Polish and perfect. 

\begin{rmk}
 Let $K$ be a Cantor set and $(\V,\dist)$ Polish and perfect. 
 Define 
 \[
  \Emb(K,\V)=\{h\in C(K,\V):f\text{ is injective}\}.
 \]
Consider $\U_n$ a sequence of clopen partitions of $K$ such that $\mesh(\U_n)\to 0$ and let 
\[
 U_n=\{h\in C(K,\V):h(A)\cap h(B)=\emptyset\text{ if }A,B\in \U_n, A\neq B\}.
\]
It is easy to see that each $U_n$ is open (and dense), and that 
$\cap_{n\in\N}U_n=\Emb(K,\V)$. 
Consequently, $\Emb(K,\V)$ is a $G_\delta$ subset of $C(K,\V)$, and by Theorem \ref{thmAlex} it is a Polish space.
\end{rmk}

\begin{thm}
\label{thmGenDSSH} 
If $(\V,\dist)$ is Polish and perfect then 
\[
\DS^{SH}(\V)= \{f\in\DS(\V): f\text{ is conjugate to the Special Homeomorphism}\} 
\]
is a dense $G_\delta$ subset of $\DS(\V)$.
\end{thm}

\begin{proof}
Let $K\subset\V$ be a fixed Cantor space.
Consider the map 
$$\varphi\colon \Emb(K,\V)\times H(K)\to\DS_{CH}(\V)$$ 
defined as $\varphi(h,f)=h f h^{-1}.$ 
Let $$\G=\{f\in H(K):f\text{ is conjugate to the Special Homeomorphism}\}.$$
Define 
\[
 R=\Emb(K,\V)\times \G.
\]
By \cites{KeRo,AGW}, we have that $R$ is a dense $G_\delta$ subset of $\Emb(K,\V)\times H(K)$. 
Notice that $\varphi(R)=\DS^{SH}(\V)$ and $\varphi^{-1}(\varphi(R))=R$.
Therefore, in order to finish the proof we need to check the remaining hypothesis of Corollary \ref{corTopFubini}.

It is clear that $\varphi$ is continuous and surjective. 
We will show that $\varphi$ is open. 
For this purpose, fix $h\in\Emb(K,\V)$ and $f\in H(K)$. 
Let $U\subset \Emb(K,\V)\times H(K)$ be an open set containing $(h,f)$ and take $\epsilon>0$ such that 
if $\dist_{C^0}(f,f')<\epsilon$ and $\dist_{C^0}(h,h')<\epsilon$ then $(h',f')\in U$. 
Let $g\colon h(K)\to h(K)$ be the homeomorphism $g=\varphi(h,f)$.
As $h^{-1}$ is uniformly continuous, there is $\delta>0$ such that 
\begin{equation}
\label{ecuCota}
  \text{if }x,y\in h(K), \dist(x,y)<2\delta\text{ then }\dist(h^{-1}(x),h^{-1}(y))<\epsilon.
\end{equation}
Suppose that $j\in\DS(\V)$ is a homeomorphism of a Cantor space 
such that $\D(g,j)<\delta$. 
By Proposition \ref{propCantorCercaDos}, consider $h_1,h_2$ with
$\|l_i\|<\delta$. 
Let $h'=h_1  h\in\Emb(K,\V)$ and $f'=h^{-1}  h_1^{-1}  j  h_1  h\in H(K)$.
Since $\|h_1\|<\delta$ we have that $\dist_{C^0}(h,h')<\delta$. 
Also
\[
 \dist_{C^0}(g,h_1^{-1}jh_1)=\dist_{C^0}(g,h_1^{-1}h_2g)<2\delta
\]
and 
\[
 \dist_{C^0}(gh,h_1^{-1}jh_1h)<2\delta.
\]
This and \eqref{ecuCota} implies that 
\[
 \dist_{C^0}(h^{-1}gh,h^{-1}h_1^{-1}jh_1h)<\epsilon
\]
and $\dist_{C^0}(f,f')<\epsilon.$	
Thus, $(h',f')\in U$. Since 
$\varphi(h',f')=j$ we conclude that $\varphi$ is open and the proof ends.
% 
% The result follows by Theorem \ref{thmTopFubini}.
\end{proof}

The next result is an application of the densitity of the conjugacy class 
of the Special Homeomorphism. This kind of result was suggested to the author by Mauricio Achigar.

\begin{cor}
\label{corMani}
 If $f\colon M\to M$ is a homeomorphism of a compact manifold $(M,\dist)$ then 
 for all $\epsilon>0$ there are an $\epsilon$-dense Cantor set $K\subset M$ and a 
 homeomorphism $g\colon K\to K$ conjugate to the Special Homeomorphism 
 such that $\dist_{C^0}(g,f|_K)<\epsilon$. 
%  \aca se pueden encajar otras propiedades a la $g$, como que sea un subshift de tipo finito...
\end{cor}

\begin{proof}
Given $\epsilon>0$, take $\delta\in(0,\epsilon/2)$ such that if 
$x,x'\in M$ and $\dist(x,x')<\delta$ then $\dist(f(x),f(x'))<\epsilon/2$. 
Applying Theorem \ref{thmGenDSSH} with $\V=M$, we have that there is $g\in\DS(M)$ conjugate to the Special Homeomorphism with 
 $\D(f,g)<\delta$. 
 Let $K=M(g)$. 
 Then, Remark \ref{rmkDistsChotas}, 
 $\dist_H(K,M)<\D(f,g)<\delta<\epsilon$, where $M=M(f)$, i.e., $K$ is $\epsilon$-dense in $M$. 
 
Since $\D(f,g)<\delta$, for all $x\in K$ there is $x'\in M$ such that 
$\dist(x,x')<\delta$ and $\dist(g(x),f(x'))<\delta$. 
By the uniform continuity of $f$, we have that $\dist(f(x),f(x'))<\epsilon/2$. 
Thus, 
$$
\dist(g(x),f(x))\leq\dist(g(x),f(x'))+\dist(f(x'),f(x))<\delta+\epsilon/2<\epsilon
$$
and $\dist_{C^0}(f,g)<\epsilon$.
\end{proof}

\section{Dynamics modulo isometric conjugacies}
\label{secDynGH}
In this section we study the equivalence relation of isometric conjugacy 
between dynamical systems in $\DS(\V)$. 
In \S\ref{secQuotSD} we apply Theorem \ref{thmGenDSSH} to conclude that the Special Homeomorphism 
is also generic in this setting.
In \S\ref{secAM} we show that our approach is topologically equivalent to \cite{AM}.

\subsection{Metric spaces}
\label{secSS}

Given two metric spaces $(X,\dist^X)$ and $(Y,\dist^Y)$ a surjective function 
$\varphi\colon X\to Y$ is an \emph{isometry} 
if $\dist^Y(\varphi(x),\varphi(x'))=\dist^X(x,x')$ for all 
$x,x'\in X$. 
In this case we write $X\sim Y$ and we say that $Y$ is an \emph{isometric copy} of $X$. 
We say that a metric space $(X,\dist)$ is \emph{homogeneous} if for all $x,y\in X$ there is an isometry $\varphi\colon M\to M$ such that $y=\varphi(x)$.

A metric space is \emph{universal} if it contains an isometric copy of every 
separable
metric space. 
An example of a universal space is $\F=l^\infty$, the Banach space 
of bounded real sequences $a\colon \N\to \R$ with the sup-norm
$\|a\|_\infty=\sup_{n\in\N}|a_n|$.
Indeed, given a separable metric space $(M,\rho)$ with a dense sequence $\{c_n\}_{n\in\N}$ 
consider $\varphi\colon M\to\F$ as $\varphi(x)=a$ where $a_n=\rho(x,c_n)-\rho(c_n,c_1)$. 
The set $\varphi(M)\subset \F$ is an isometric copy of $M$.
The function $\varphi$ is known as a \emph{Fr\'echet embedding}.
Since the translations of $\F$ are isometries, we have that $\F$ is homogeneous.
The space $\F$ is complete but not separable. 
Another classical universal space is 
the Banach space of continuous functions from $[0,1]$ to $\R$ with the sup-norm, see \cite{Heinonen}.

% \begin{rmk}
%  A universal space may not be homogeneous. 
%  First note that if $\V_1$ is universal and $\V_2$ contains an isometric copy of $\V_1$ then $\V_2$ is universal. 
%  Thus, we can take $\V=\F\cup \{x\}$, for some $x\notin\F$, to obtain a non-homogeneous universal space.
% \end{rmk}

% Let us show that a space containing a copy of each compact metric space may not be universal.
% 
% \begin{ex}
% Let $\V=\{(n,a)\in\N\times \F:\|a\|_\infty\leq n\}$. 
% On $\V$ consider the metric $\dist((n,a),(n',a'))=\max\{|n-n'|,\|a-a'\|_\infty\}$.
% It contains an isometric copy of each compact metric space but it is not universal (as it does not contain an 
% isometric copy of $\R$).
% \end{ex}

A metric space $(\V,\dist)$ is \emph{ultrahomogeneous} if any isometry $\varphi\colon A\to B$ between two
compact subsets of $\V$, extends to an isometry $\varphi'\colon \Ur\to\Ur$.
We remark that Euclidean $\R^n$ is not universal but it is Polish, perfect and ultrahomogeneous (see \cite{Me08}*{Corollary 1}).
A metric space is \emph{Urysohn universal} if it is 
ultrahomogeneous, universal and Polish. 
The next result is due to Urysohn.

\begin{thm}[\cites{Heinonen,Gro}]
Up to isometry,
there is a unique Urysohn universal space. 
\end{thm}

\begin{rmk}
The ultrahomogeneous condition is usually stated for $A,B$ finite sets, but from \cites{Huhu,Hol} 
the formulation given above is equivalent. See also \cite{Gro}*{Exercise (b'), p. 82}.
\end{rmk}

\emph{The} Urysohn universal space will be denoted as $(\Ur,\dist)$. 
It is clear that $\Ur$ is perfect.
The interested reader may see \cite{Nguyen}*{p. 32} for more on ultrahomogeneous spaces.

For a metric space $(\V,\dist)$, the equivalence relation of isometry $\sim$ on the space $(\K(\V),\dist_H)$ gives rise 
to a quotient space 
\[
 \K_{\Iso}(\V)=\K(\V)/\sim.
\]
For $X\in \K(\V)$ denote by $\tilde X=\{Y\in\K(\V):X\sim Y\}$ its equivalence class.
Given $X,Y\in\K(\V)$ define 
\begin{equation}
 \label{ecuGHCociente}
  \dist_{\tilde H}(\tilde X,\tilde Y)=\inf\{\dist_H(A,B):A\sim X, B\sim Y, A,B\in\K(\V)\}.  
\end{equation}
This is the so called \emph{Gromov-Hausdorff distance}. 
As we learned in \cite{Tuz}, it was first defined by Edwards \cite{Ed}.
The next result summarizes some standard properties. 

\begin{prop}[\cites{BBI,Gro,Ed}]
If $(\V,\dist)$ is complete and ultrahomogeneous then
$\dist_{\tilde H}$ is a complete metric in $\K_{\Iso}(\V)$ 
compatible with the quotient topology.
\end{prop}

It is interesting to remark how these metrics are related for $\V=\R^n,\Ur$. 
Denote by $\dist^{\R^n}$ and $\dist^\Ur$ the Euclidean metric and the metric of Urysohn universal space, respectively.
In \cite{Me08}*{Theorem 2} it is shown that 
\[
 \dist^{\Ur}_{\tilde H}(X,Y)\leq\dist^{\R^n}_{\tilde H}(X,Y)\leq c_n\sqrt{\dist^{\Ur}_{\tilde H}(X,Y)}
\]
for all $X,Y\in\K(\R^n)$. The constants $c_n$ depends only on the dimension of $\R^n$.
Naturally, to compute $\dist^{\Ur}_{\tilde H}(X,Y)$ we are taking isometric copies in $\Ur$.

\subsection{Isometric conjugacies}
\label{secIsoConj}
Let $(\V,\dist)$ be a metric space. 
We define an equivalence relation on $\DS(\V)$ as 
$f\sim g$ if there is an isometry $\varphi\colon M(f)\to M(g)$ such that 
$\varphi^{-1}  g  \varphi=f$. 
In this case we say that $\varphi$ is an \emph{isometric conjugacy}.
% \begin{rmk}
% If $f,g\in\DS(\V)$ are functions then $f\sim g$ means that $f$ and $g$ are conjugate by an isometry.  
% \end{rmk}

\begin{rmk}
 Suppose that $f\sim g$ with an isometry $\varphi\colon M(f)\to M(g)$. 
 Let $\psi\colon M(f)\times M(f)\to M(g)\times M(g)$ be the isometry $\psi(x,x')=(\varphi(x),\varphi(x'))$. 
 We will show that $\psi(f)=g$. 
 We know that $(x,y)\in f$ iff $(x,y)\in \varphi^{-1}  g  \varphi$, 
 that is $(\varphi(x),\varphi(y))\in g$, i.e. $\psi(x,y)\in g$.  
 This implies that $f$ and $g$ are isometric.
\end{rmk}

\begin{rmk}
\label{rmkGraficasChotas}
 Two maps $f,g\in\DS$ may have isometric graphs but they may not be conjugate. 
 In Figure \ref{figGraficas} we see an example. 
 Note that $f$ and $f^{-1}$ are isometric sets for all $f\in\DS(\V)$.
 \begin{figure}[h]
 \begin{center}
  \includegraphics{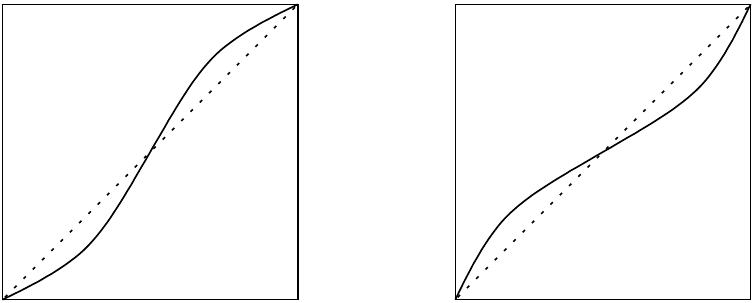}
  \caption{Two homeomorphisms $f$ and $f^{-1}$ of an interval. 
  On the left, there are 3 fixed points: a repeller and two attractors. 
  On the right: one attractor and two repellers. Thus, these homeomorphisms are not conjugate but their graphs are isometric.}
  \label{figGraficas}
 \end{center}
 \end{figure}
\end{rmk}

\subsection{The quotient space}
\label{secQuotSD}
Consider the quotient of $\DS(\V)$ by isometric conjugacies 
$$\DS_{\Iso}(\V)=\frac{\DS(\V)}{\sim}$$ with the quotient topology and 
denote by $\pi\colon \DS(\V)\to\DS_{\Iso}(\V)$ the canonical projection. 
Recall from \S \ref{secDS}
that $\D$ is the Hausdorff distance in $\K(\V^2)$.
Define 
\[
 \DGH(\tilde f,\tilde g)=\inf \D(f',g')
\]
where $\inf$ is taken over all $f'\sim f$ and $g'\sim g$.
It is clear that 
\begin{equation}
 \label{ecuDistOrto}
\D(f,g)\geq \DGH(\tilde f,\tilde g)
\end{equation}
for all $f,g\in\DS(\V)$.

\begin{prop}
\label{propPiAbierta}
 If $(\V,\dist)$ is ultrahomogeneous then: 
 \begin{enumerate}
  \item $\DGH$ is a metric compatible with the quotient topology,
  \item it holds that $\pi(B_r(f))=B_r(\tilde f)$ for all $r>0$, $f\in\DS(\V)$ and $\pi$ is an open map,
  \item If in addition $(\V,\dist)$ is complete then $(\DS_{\Iso}(\V),\DGH)$ is complete.
 \end{enumerate}
\end{prop}

\begin{proof}
First we show that $\pi(B_r(f))=B_r(\tilde f)$. 
By \eqref{ecuDistOrto} we have that $\pi(B_r(f))\subset B_r(\tilde f).$ 
To prove the other inclusion suppose that $\DGH(\tilde f, \tilde g)<r$.
By definition, there are $f'\sim f$ and $g'\sim g$ such that $\D(f',g')<r$. 
Thus, there is an isometry $\varphi \colon M(f')\to M(f)$ conjugating $f$ and $f'$. 
As $\V$ is ultrahomogeneous, there is an isometry $\psi\colon\V\to\V$ with $\psi|_{M(f')}=\varphi$. 
Let $g''\in \DS(\V)$ be defined as 
 $g''=\{(\psi(x),\psi(y))\in\V^2:(x,y)\in g'\}$. 
Since $\psi$ is an isometry, we have that $\D(f,g'')=\D(f',g')<r$ and $g''\sim g$.
Therefore, $B_r(\tilde f)\subset \pi(B_r(f))$.

To prove that $\DGH$ is a metric, note that $\DGH$ is non-negative and symetric. 
The triangular inequality follows from the ultrahomogeneity of $\V$. 
If $\DGH(\tilde f,\tilde g)=0$ then by  
Ascoli Theorem \cite{Kelley}*{p. 234} we have that $f\sim g$.

By \eqref{ecuDistOrto} we have that $\pi$ is continuous with respect to $\DGH$, which implies that 
every open set $U\subset\DS_{\Iso}(\V)$ is open in the quotient topology. 
Conversely, suppose that $U$ is open in the quotient topology. 
Then, $\pi^{-1}(U)$ is open in $\DS(\V)$. 
For each $f\in \pi^{-1}(U)$ there is $r>0$ such that $B_r(f)\subset \pi^{-1}(U)$. 
Thus, $\pi(B_r(f))=B_r(\tilde f)\subset U$. This proves that $U$ is open with respect to $\DGH$. 

To prove that $\DS_{\Iso}(\V)$ is complete, let $\tilde f_n\in\DS_{\Iso}(\V)$ be a Cauchy sequence. 
Take a subsequence $\tilde f_{n_k}$ such that 
$\sum_{k=1}^\infty\DGH(\tilde f_{n_k},\tilde f_{n_{k+1}})<\infty$. 
Since $\V$ is ultrahomogeneous there is $g_k\in\DS(\V)$ such that 
$g_k\sim f_{n_k}$ and 
$\D(g_k,g_{k+1})<2\DGH(\tilde f_{n_k},\tilde f_{n_{k+1}})$ for all $k\in\N$. 
In this way, $g_k$ is a Cauchy sequence. 
Since $\V$ is complete, by Proposition \ref{propDSComplete} 
we have that $\DS(\V)$ is complete and $g_k$ is convergent with limit $g\in\DS(\V)$. 
Since $\pi$ is continuous, $\tilde g_k=\tilde f_{n_k}$ converges to 
$\tilde g$. 
As $\tilde f_n$ is a Cauchy sequence with a convergent subsequence, we conclude that 
it is convergent and the proof ends.
\end{proof}

\begin{thm}
\label{thmGHGen}
If $(\V,\dist)$ is Polish, perfect and ultrahomogeneous
then 
\[
 \DS^{SH}_{\Iso}(\V)=\{\tilde f\in \DS_{\Iso}(\V): f \text{ is conjugate to the Special Homeomorphism}\} 
\]
is a dense $G_\delta$ subset of $\DS_{\Iso}(\V)$.
\end{thm}

\begin{proof}
The projection $\pi\colon\DS(\V)\to\DS_{\Iso}(\V)$ is continuous, surjective and open (Proposition \ref{propPiAbierta}).
By Theorem \ref{thmGenDSSH} we have 
that 
$\DS^{SH}(\V)$ is a dense $G_\delta$ subset of $\DS(\V)$.
Note that $\pi(\DS^{SH}(\V))=\DS^{SH}_{\Iso}(\V)$ and $\pi^{-1}(\pi(\DS^{SH}(\V)))=\DS^{SH}(\V)$.
Therefore, the result follows by Corollary \ref{corTopFubini}.
\end{proof}

\subsection{Arbieto-Morales metric}
\label{secAM}
Given two metric spaces $(X,\dist^X)$ and $(Y,\dist^Y)$ a function 
$\varphi\colon X\to Y$ is an $\epsilon$-\emph{isometry} if 
\[
 |\dist^Y(\varphi(x),\varphi(x'))-\dist^X(x,x')|<\epsilon
\]
for all $x,x'\in X$ and $\dist^Y_H(\varphi(X),Y)<\epsilon$.
Let $\Iso_\epsilon(X,Y)$ be the set of $\epsilon$-isometries from $X$ to $Y$.
The $C^0$-\emph{Gromov-Hausdorff distance} \cite{AM} 
between the maps $f\colon X\to X$ and $g\colon Y\to Y$ is defined as 
\begin{equation}
 \label{eqDefAM}
\begin{array}{rl}
 \DAM(f,g)=& \inf\{\epsilon>0: 
 \exists \varphi\in\Iso_\epsilon(X,Y)
 \text{ and }
 \psi\in\Iso_\epsilon(Y,X)
 \text{ s.t.}\\
 & \dist_{C^0}(g  \varphi,\varphi  f)<\epsilon\text{ and }
 \dist_{C^0}(\psi  g,f  \psi)<\epsilon\}.
\end{array}
\end{equation}
Recall $\dist_{C^0}$ from \eqref{ecuDistC0}. 
Note that $\DAM$ is defined for maps and not for arbitrary relations. 

\begin{rmk}[\cite{AM}*{Theorem 1}]
\label{rmkAMmetrica}
 The $C^0$-Gromov-Hausdorff distance satisfies the following properties: 
 \begin{enumerate}
  \item $\DAM(f,g)=0$ if and only if $f\sim g$, 
  \item $\DAM(f,g)=\DAM(g,f)$ and 
  \item $\DAM(f,g)\leq 2(\DAM(f,h)+\DAM(h,g))$,
 \end{enumerate}
 for all continuous maps $f,g,h$ of compact metric spaces. 
 In \cite{AM} a function as $\DAM$ satisfying these conditions is called pseudo quasi-distance with coefficient 2 (in the 
 triangular inequality). 
 Notice that $\DAM$ is invariant by isometric conjugacies, i.e., if $f\sim f'$ and $g\sim g'$ then 
 $\DAM(f,g)=\DAM(f',g')$.
\end{rmk}

\begin{lem}[\cite{BBI}*{Corollary 7.3.28}]
\label{lemExtMet}
 If $\varphi\colon X\to Y$ is an $\epsilon$-isometry 
 then there is a metric $\dist^Z$ in the disjoint union $Z=X\cup Y$ such that 
 $\dist^Z(x,x')=\dist_X(x,x')$ for all $x,x'\in X$, 
 $\dist^Z(y,y')=\dist_Y(y,y')$ for all $y,y'\in Y$, 
 $\dist^Z_H(X,Y)\leq 2\epsilon$ and 
 $\dist(x,\varphi(x))=\epsilon$ for all $x\in X$.
\end{lem}

% See  for a proof of Lemma \ref{lemExtMet}. 
Since in \cite{AM} the authors consider dynamical systems on all the compact metric spaces, 
the next proof is only given for $\V=\Ur$.

\begin{thm}
\label{thmCompAm}
The restriction of $\DGH$ to continuous maps in $\DS_{\Iso}(\Ur)$ is 
topologically equivalent to $\DAM$.
\end{thm}

\begin{proof}
First we will show that 
\begin{equation}
 \label{ecuDAMDGH}
\DAM(f,g)\leq 2\DGH(\tilde f,\tilde g)
\end{equation}
Suppose that $\DGH(\tilde f,\tilde g)<\delta$. 
Taking isometric conjugacies we can assume that $\D(f,g)<\delta$. 
For each $x\in X=M(f)$ there is $y\in Y=M(g)$ such that 
$\dist(x,y)<\delta$ and $\dist(f(x),g(y))<\delta$. 
This defines a function $\varphi \colon X\to Y$ as $\varphi (x)=y$ 
satisfying $\dist(x,\varphi (x))<\delta$ and $\dist(f(x),g(\varphi (x)))<\delta$ for all 
$x\in X$. 

We will show that $\varphi \in\Iso_{2\delta}(X,Y)$.
Since $\D(f,g)<\delta$, given $y\in Y$ there is $x\in X$ such that 
$\dist(x,y)<\delta$ and $\dist(f(x),g(y))<\delta$. 
Then $$\dist(y,\varphi (x))\leq \dist(y,x)+\dist(x,\varphi (x))<\delta+\delta=2\delta.$$
This proves that $\dist_H(\varphi (X),Y)<2\delta$. 
Given $x,x'\in X$ note that 
\[
|\dist(\varphi (x),\varphi (x'))-\dist(x,x')|\leq \dist(\varphi (x),x)+\dist(x',\varphi (x'))<2\delta.
\]
Thus, $\varphi \in\Iso_{2\delta}(X,Y)$.
Given $x\in X$ we have that
\[
 \dist(g(\varphi (x)),\varphi (f(x)))\leq \dist(g(\varphi (x)),f(x))+\dist(f(x),\varphi (f(x)))<2\delta.
\]
This proves that $\dist_{C^0}(g  \varphi ,\varphi   f)<2\delta$. 
In a similar way we can define the required $\psi\in\Iso_{2\delta}(Y,X)$. 
This proves \eqref{ecuDAMDGH}.

Given a map $g\in\DS(\Ur)$ and $\epsilon>0$ take 
$\delta\in (0,\epsilon/3)$ such that if $u,v\in M(g)$ and $\dist(u,v)<\delta$ then $\dist(g(u),g(v))<\epsilon/3$.
Suppose that $\DAM(f,g)<\delta$.
By definition, there are
$\varphi\in\Iso_\delta(X,Y)$
and
$\psi\in\Iso_\delta(Y,X)$
such that $\dist_{C^0}(g  \varphi,\varphi  f)<\delta$ and 
 $\dist_{C^0}(\psi  g,f  \psi)<\delta$.
By Lemma \ref{lemExtMet}, we can (in addition) assume that 
$\dist(x,\varphi(x))<\delta$
for all $x\in X$.
Then
\[
\dist(f(x),g(\varphi(x)))\leq\dist(f(x),\varphi(f(x)))+\dist(\varphi(f(x)),g(\varphi(x)))
 <\delta+\delta=2\delta.
\]
Since $\dist^Y_H(\varphi(M(f),M(g))<\delta$, 
for all $y\in Y$ there is $x\in X$ such that $\dist(y,\varphi(x))<\delta$. 
Then
\[
 \dist(x,y)\leq\dist(x,\varphi(x))+\dist(\varphi(x),y)<\delta+\delta=2\delta.
\]
Also
\begin{align*}
 \dist(f(x),g(y)) \leq &\dist(f(x),\varphi(f(x)))+\dist(\varphi(f(x)),g(\varphi(x)))\\
  &+\dist(g(\varphi(x)),g(y))\\
 \leq& \delta+\delta+\epsilon/3<\epsilon.
\end{align*}
As explained in Remark \ref{rmkDistDS}, this proves that $\D(f,g)<\epsilon$.
\end{proof}

\end{document}